\documentclass[
final,
]{amsart}

\usepackage[english]{babel}
\usepackage[utf8]{inputenc}

\usepackage[T1]{fontenc}
\usepackage{lmodern}
\usepackage{hyperref}

\usepackage[inline]{enumitem}
\usepackage{verbatim}
\usepackage{url}
\usepackage[notref,notcite,color]{showkeys}
\usepackage{microtype}
\usepackage{varwidth} 
\usepackage{mathtools}
\usepackage[all]{xy}

\usepackage[LogicalSystemsSans]{ak_logic}

\newcommand*{\Theorem}{Theorem}
\newcommand*{\Proposition}{Proposition}
\newcommand*{\Lemma}{Lemma}
\newcommand*{\Corollary}{Corollary}
\newcommand*{\Definition}{Definition}
\newcommand*{\Question}{Question}
\newcommand*{\Remark}{Remark}
\newcommand*{\Notation}{Notation}

\theoremstyle{plain}
\newtheorem{theorem}{\Theorem}
\newtheorem{proposition}[theorem]{\Proposition}
\newtheorem{corollary}[theorem]{\Corollary}
\newtheorem{lemma}[theorem]{\Lemma}
\theoremstyle{definition}
\newtheorem{definition}[theorem]{\Definition}
\newtheorem{question}{\Question}
\theoremstyle{remark}
\newtheorem{remark}[theorem]{\Remark}

\usepackage{prettyref}
\newrefformat{thm}{\Theorem~\ref{#1}}
\newrefformat{pro}{\Proposition~\ref{#1}}
\newrefformat{cor}{\Corollary~\ref{#1}}
\newrefformat{lem}{\Lemma~\ref{#1}}
\newrefformat{rem}{\Remark~\ref{#1}}
\newrefformat{def}{\Definition~\ref{#1}}

\renewcommand{\epsilon}{\varepsilon}

\newcommand{\U}{\mathcal{U}}
\newcommand{\Uidem}{\U_{\textup{\rm idem}}}

\newcommand{\Umin}{\U_{\textup{\rm min}}}
\newcommand{\F}{\mathcal{F}}

\newcommand{\FS}[1]{{\mathrm{FS}(#1)}}

\DeclareMathOperator*{\iplim}{\textup{IP-lim}}
\newcommand{\ulim}[2]{{#1}\textup{-}\!\lim_{#2}}
\newcommand{\orbc}{\overline{\textsl{Orb}}}
\DeclareMathSymbol{\two}{\mathalpha}{letters}{`2}

\title{Minimal idempotent ultrafilters and the Auslander-Ellis theorem}

\author{Alexander P.\ Kreuzer}
\address{Department of Mathematics \\
Faculty of Science \\
National University of Singapore \\
Block S17, 10 Lower Kent Ridge Road \\
Singapore 119076 
}
\email{matkaps@nus.edu.sg}
\urladdr{\url{http://www.math.nus.edu.sg/~matkaps/}}
\subjclass[2010]{Primary: 03B30, Secondary: 03F35, 54H20, 05D10}
\keywords{Reverse mathematics, special ultrafilters, central sets}

\thanks{I thank Vassilis Gregoriades for useful discussions.}
\thanks{The author was partly supported by the RECRE project and by the Ministry of Education of Singapore through grant R146-000-184-112 (MOE2013-T2-1-062).}

\begin{document}

\begin{abstract}
  We characterize the existence of minimal idempotent ultrafilters (on $\Nat$) in the style of reverse mathematics and higher-order reverse mathematics using the Auslander-Ellis theorem and a variant thereof.

  We obtain that the existence of minimal idempotent ultrafilters restricted to countable algebras of sets is equivalent to the Auslander-Ellis theorem (\lp{AET}) and that the existence of minimal idempotent ultrafilters as higher-order objects is $\Pi^1_2$-conservative over a refinement of \lp{AET}.
\end{abstract}

\maketitle

We characterize the existence of minimal idempotent ultrafilters (on $\Nat$) in the style of reverse mathematics and higher-order reverse mathematics using the Auslander-Ellis theorem and a variant thereof.

The set of all ultrafilters on $\Nat$ can be identified with the Stone-\v{C}ech compactification $\beta\Nat$ of $\Nat$. The natural number are represented in $\beta\Nat$ as the principal ultrafilters via the embedding
\[
n\longmapsto \{\, X\subseteq \Nat \mid n\in X \,\}\in \beta\Nat.
\]
One can show that one can extend the addition of $\Nat$ to $\beta\Nat$ in the following way.
\[
\U + \mathcal{V} = \{\, X \subseteq \Nat \mid \{\,n\in \Nat \mid X-n \in \mathcal{V} \,\} \in \mathcal{U}\,\}
.\]
An ultrafilter $\U$ is called \emph{minimal} if one of the following equivalent conditions holds:
\begin{enumerate}[label=(\arabic*)]
\item $\U$ belongs to a minimal ideal of $(\beta\Nat,+)$,
\item $(\U + \beta\Nat, \sigma)$ with $\sigma\colon \U \longmapsto \U + 1$ is a minimal dynamical system (with respect to the usual topology on $\beta\Nat$ generate by the basic open sets $B(X) := \{ \mathcal{V}\in \beta \Nat \mid X\in \mathcal{V} \}$),
\item\label{enum:minsyn} For each $X\in \U$ the set $\{ n\in \Nat \mid X-n \in \U\}$ is syndetic.
  (Recall that a set $X\subseteq \Nat$ is called \emph{syndetic} if there is an $m$ such that for each $x$ we have $X\cap [x,x+m] \neq \emptyset$.)
\end{enumerate}
(see Hindman, Strauss \cite{HS12}).

We will use \ref{enum:minsyn} as definition for minimal ultrafilters, since it does not refer to any subsets of the Stone-\v{C}ech compactification and can, therefore, be expressed with the lowest quantifier complexity.

An ultrafilter $\U$ is called \emph{idempotent} if $\U=\U+\U$ and \emph{minimal idempotent} if it is minimal and idempotent.

Our interest in minimal idempotent ultrafilters stems from the fact that they are widely used in ergodic theory and combinatorics, see e.g.\ Carlson, Simpson \cite{CS84}, Carlson \cite{tC88}, Gowers \cite{wG92}, and Bergelson \cite{vB10}.

First we will restrict our attention to countable collections of sets as Hirst did in his analysis of idempotent ultrafilters in \cite{jH04}. Here, we will show that the statement that for a countable collection of sets a minimal idempotent ultrafilter restricted to this collection exists is equivalent  to the Auslander-Ellis theorem (\lp{AET}) over \ls{RCA_0}.

Then we will analyze the existence of minimal idempotent ultrafilter in a higher-order system (\ls{RCA_0^\omega}) building on previous work in \cite{aK12b,aK15}. We will see that the $\Pi^1_2$\nobreakdash-consequences are equivalent to refinements of the Auslander-Ellis theorem (\lp{eAET}, \lp{eAET'}, depending on the precise formulation of the existence of the ultrafilter). Beside of this the strength of \lp{eAET},\lp{eAET'} remains unknown.

The paper is organized in the following way. In \prettyref{sec:aet} we will recall the Auslander-Ellis theorem and define \lp{eAET} and \lp{eAET'}, in \prettyref{sec:mincount} we will define and analyze the statement that a minimal idempotent ultrafilter restricted to a countable algebra exists, and in \prettyref{sec:minhigh} we will analyze the general case in the higher-order setting.

\section{The Auslander-Ellis theorem}\label{sec:aet}

\begin{definition}\mbox{}
  \begin{itemize}
  \item Let $(\mathcal{X},d)$ be a compact metric space and let $T\colon \mathcal{X}\longrightarrow \mathcal{X}$ be a continuous mapping. Then we call $(\mathcal{X},T)$ a \emph{compact topological dynamical system}.
  \item A point $x$ in  $(\mathcal{X},T)$ is called \emph{uniformly recurrent} if for each $\epsilon > 0$ the set $\{ n\in \Nat \mid d(T^n x,x) < \epsilon \}$ is syndetic.
  \item A pair of points $x,y$ in $(\mathcal{X},T)$ is called \emph{proximal} if for each $\epsilon > 0$ there are infinitely many  $n$ with $d(T^n x, T^n y) < \epsilon$. 
  \end{itemize}
\end{definition}

\begin{definition}
  The \emph{Auslander-Ellis theorem} (\lp{AET}) is the statement that for each compact topological dynamical system $(\mathcal{X},T)$ and each point $x\in \mathcal{X}$ there exists a uniformly recurrent point $y$ such that $x,y$ are proximal.
\end{definition}

Recall that \emph{Hindman's theorem} (\lp{HT}) is the statement that for each coloring $c\colon \Nat \longrightarrow 2$ of the natural numbers there exists an infinite set $X=\{x_0,x_1,\dots\}$ such that the set of finite sums of $X$ 
\[
\FS{X}= \FS{(x_i)_i} := \left\{\, x_{i_k} + \dots + x_{i_1} \sizeMid i_1 < i_2 < \dots < i_k \,\right\}
\]
is homogeneous for $c$. The \emph{iterated Hindman's theorem} (\lp{IHT}) is the statement that for each sequence of colorings $c_k\colon \Nat \longrightarrow 2$ there exists a strictly ascending sequence $(x_i)_{i\in\Nat}$ such that for each $k$ the set $\FS{(x_i)_{i=k}^\infty}$ is homogeneous for $c_k$. It is known that \lp{HT}, \lp{IHT} are provable in \lp{ACA_0^+} and imply \lp{ACA_0}. However it is open where between these systems \lp{HT} and \lp{IHT} lie, and whether they are equivalent.

Blass, Hirst, and Simpson showed in \cite{BHS87} that \lp{AET} is provable in $\lp{ACA_0}+\lp{IHT}$ and thus in \lp{ACA_0^+}.
In fact, \lp{AET} is equivalent to \lp{IHT} over \ls{RCA_0}, see \prettyref{cor:aetiht} below.

We will write  
\[
\iplim_{n\to \FS{(n_i)}} x_n = x
,\]
for the statement for each $\epsilon$ there exists a $k$ such that for all finite sums $n \in \FS{(n_i)_{i\ge k}}$ we have $d(x_n,x) < \epsilon$.

\begin{proposition}[\ls{RCA_0}]\label{pro:ip}
  For $(\mathcal{X},T)$, $x,y$ as in the Auslander-Ellis theorem the system \ls{RCA_0} proves that there exists an increasing sequence $(n_i)_{i\in\Nat}$ such that
  \[
  \iplim_{n\to \FS{(n_i)}} T^n x = y
  .\]
\end{proposition}
To prove this proposition we need the following lemmata. For this fix $(\mathcal{X},T)$.
We will denote by $\orbc(y)$ the orbit closure of $y$, i.e.~the set 
\[
\left\{\, x\in\mathcal{X} \sizeMid \Forall{\epsilon>0}\Exists{n}\, d(T^n y,x)<\epsilon \,\right\}
.\]
\begin{lemma}[\ls{RCA_0}]\label{lem:min}
  If $y$ is uniformly recurrent, then for each $z\in \orbc(y)$ we have $\orbc(y)=\orbc(z)$. In other words, $\orbc(y)$ is minimal.
\end{lemma}
\begin{proof}
  It is clear that $\orbc(z) \subseteq \orbc(y)$. Suppose that $\orbc(z) \neq \orbc(y)$. Then there exists an $\epsilon$ such that each point in $\orbc(z)$ is more than $3\epsilon$ apart from $y$.
  Since $y$ is uniformly recurrent we have an $m$ such that in $(T^n(y))_n$ at least every $m$-th element is $\epsilon$-close to $y$.
  However, there is also a sequence $(n_k)$ such that $T^{n_k}y \longrightarrow z$.
  By continuity of $T$, we can find an $\epsilon'$ such that for each $z'$ that is $\epsilon'$\nobreakdash-close to $z$ the iterates $Tz',T^2z',\dots,T^mz'$ are all $\epsilon$-close to $Tz$ resp.\ $T^2z,\dots,T^mz$.
  Since $Tz,T^2z,\dots,T^mz\in \orbc(z)$ these elements have a distance of $3\epsilon$ to $y$ and therefore $Tz',T^2z',\dots,T^mz'$ are at least $2\epsilon$ apart from $y$. 
  Now choosing a $k$ such that $d(T^{n_k}y,z)<\epsilon'$ yields that $T^{1+n_k}y,T^{2+n_k}y,T^{m+n_k}y$ are all $2\epsilon$ apart from $y$ contradicting the fact that $y$ is uniformly recurrent.
\end{proof}

\begin{lemma}[\ls{RCA_0}]\label{lem:forb}
  If $y$ is uniformly recurrent then for each open set $U$ in $\orbc(y)$ there is an $m$ such that
  \begin{equation}\label{eq:forb}
  \bigcup_{n=0}^m T^{-n}(U) \supseteq \orbc(y)
  .\end{equation}
\end{lemma}
\begin{proof}
By \prettyref{lem:min} we known that the orbit of each $z\in \orbc(y)$ meets $U$. Therefore,
\[
\bigcup_{n=0}^\infty T^{-n}(U) \supseteq \orbc(y)
.\]
In other words the sets $T^{-n}(U)$ form an open covering of the compact set $\orbc(y)$. Using \lp{WKL} we can find a finite sub-covering and thus an $m$ such that \eqref{eq:forb} holds (see \cite[IV.1]{sS09}). 

Now \eqref{eq:forb} is equivalent to  the arithmetical statement  $\Forall{k} T^k y \in \bigcup_{n=0}^m T^{-n}(U)$. Thus, the statement of the whole lemma is $\Pi^1_1$ and therefore by Corollary~IX.2.6 of \cite{sS09} provable without \lp{WKL}.
\end{proof}

\begin{proof}[Proof of \prettyref{pro:ip}]
  We claim that for each open neighborhood $U$ of $y$ there is a $p$ such that $T^px,T^py \in U$.

  Let $\epsilon$ be such that $B(y,2\epsilon)$ is an open $2\epsilon$-ball around $y$ contained in $U$. By \prettyref{lem:forb} we can find an $m$ with
  \begin{equation}\label{eq:forb2}
  \bigcup_{n=0}^m T^{-n}(B(y,\epsilon)) \supseteq \orbc(y)
  .\end{equation}
  Now by continuity of $T$ we can find a $\delta>0$ such that for all $x',x''$ satisfying $d(x',x'')< \delta$ we have $\Forall{n\in[1,m]} d(T^nx',T^nx'') < \epsilon$. By proximality there is a $r$ such that $d(T^rx,T^ry) < \delta$ and thus $d(T^{n+r}x,T^{n+r}y)<\epsilon$ for all $n \in [1,m]$. By \eqref{eq:forb2} there is an $n \ge m$ with $d(T^{n+r} y,y)< \epsilon$ and thus $d(T^{n+r} x,y)< 2\epsilon$ and in particular $T^{n+r} x,T^{n+r} y \in U$. This finishes the proof of the claim.

  Using this claim we will construct recursively a sequence $(n_i)$ and a sequence of neighborhoods  $(U_i)$ of $y$ with the following properties:
  \begin{itemize}
  \item $U_{i+1} \subseteq U_i$, $T^{n_i} U_{i+1} \subseteq U_i$,
  \item $T^{n_i} x, T^{n_i} y \in U_{i+1}$,
  \item $U_i\subseteq B(y,2^{-i})$.
  \end{itemize}
  For this set $U_1:=B(y,2^{-1})$ and let $n_1$ be such that $T^{n_1}x,T^{n_1}y\in U_1$. Set $U_{i+1} := B(y,2^{-(i+1)}) \cap U_i \cap T^{-n_i} U_i$. This set is by induction hypothesis a neighborhood of~$y$. Let $n_{i+1}$ be such that it satisfies the claim for $U_{i+1}$.

  For a finite sum given by $n_{i_1},\dots,n_{i_k}$ with $i_1< \dots< i_k$ we then have
  \begin{align*}
    T^{n_{i_k}+\dots+n_{i_1}}(x) & \in T^{n_{i_{k-1}}+\dots+n_{i_1}}(U_{i_k}) \\ & \subseteq T^{n_{i_{k-2}}+\dots+n_{i_1}}(U_{i_{k-1}})
     \subseteq \dots \subseteq U_{i_1} \subseteq B(y,2^{-i_1})
  \end{align*}
  or in other words that $\iplim_{n\to \FS{(n_i)}} T^n x = y$.
\end{proof}
This proof is based on Proposition~8.10 of \cite{hF81}.

\begin{remark}
  \prettyref{pro:ip} should be compared with Lemma~5.3 of \cite{BHS87} which states that \lp{IHT} proves (and is in fact equivalent) to the statement that for each sequence $(x_n)_n$ in a compact space there exists an infinite set $N$ with $\iplim_{n\to \FS{N}} x_n$. 
  In the proof of \lp{AET} in \cite{BHS87} the iterated Hindman's theorem (\lp{IHT}) is used in this form (cf.~Theorem~4.13 in \cite{BHS87}). 
  One can show that this property implies that $x,y$ are proximal, see \cite[Lemma~5.2]{BHS87}. However, it is in general not the case that $y$ is uniformly recurrent.
\end{remark}

As immediate consequence we get the following corollary.
\begin{corollary}[\cite{BHS87}, Folklore]\label{cor:aetiht}
  $\ls{RCA_0} \vdash \ls{AET} \IFF \lp{IHT}$.
\end{corollary}
\begin{proof}[Sketch of proof]
  The left-to-right direction follows from the proof of \lp{AET} in \cite{BHS87}. (Theorem~4.13 in \cite{BHS87} is \lp{IHT}).

  The right-to-left direction follows from \prettyref{pro:ip} by viewing a coloring in $c_i\colon \Nat \longrightarrow 2$ as a point in the dynamical system $(2^\Nat, T)$ where $T$ is the left shift and considering the $\Nat$-fold product of this to be able to deal with all colorings. \lp{AET} together with \prettyref{pro:ip} yields then the desired IP-set. See \prettyref{thm:m} below.
\end{proof}

\begin{definition}
  The extension Auslander-Ellis theorem (\lp{eAET}) is the statement that given 
  \begin{enumerate}
  \item a compact topological dynamical system $(\mathcal{X}_1,T_1)$ with points $x_1,y_1\in\mathcal{X}_1$ satisfying the conclusion of \lp{AET}, and
  \item a second compact topological dynamical system $(\mathcal{X}_2,T_2)$ with a point $x_2\in\mathcal{X}_2$
  \end{enumerate}
  then one can find a point $y_2$ extending the solution to \lp{AET} to the product system $(\mathcal{X}_1\times \mathcal{X}_2, T_1\times T_2)$, i.e.~$\begin{pmatrix} y_1 \\ y_2 \end{pmatrix}$ is uniformly recurrent and $\begin{pmatrix} x_1 \\ x_2 \end{pmatrix},\begin{pmatrix} y_1 \\ y_2 \end{pmatrix}$ are proximal.
\end{definition}

\lp{eAET} follows from an easy adaption of any of the classical proofs of the Auslander-Ellis theorem like the original one by Ellis or Auslander (see \cite{hF81} and \cite{rE60,jA60} for a reference), a different proof by Auslander \cite{jA88}, or the proof using minimal idempotent ultrafilters (see \cite{BH90} and Chaper~19 of \cite{HS12}). 
However, it does not seem to be possible to adapt the proof of Blass, Hirst, Simpson to \lp{eAET} even though is based on the original proof as presented in \cite{hF81}.

We will also use the following consequence of \lp{eAET}.
 Let \lp{eAET'_\mathnormal{n}} be the statement that for each sequence $(t_i)_{i<n}$ of continuous functions, such that $t_i$ is an $i$\nobreakdash-ary  function $t_i\colon \mathcal{X}^i \longrightarrow \mathcal{X}$, there exists a sequence $(y_i)_{i<n}\in \mathcal{X}^n$ with $(y_i)_{i<n}$ uniformly recurrent as point in the $n$-fold product of $(\mathcal{X},T)$ and proximal to the point $(t_0,t_1(y_0),t_2(y_0,y_1),\dots,t_{n-1}(y_0,\dots,y_{n-2}))$. Let \lp{eAET'} be the union of \lp{eAET'_\mathnormal{n}}.

To see that \lp{eAET} implies \lp{eAET'} let $(t_i)_{i<n}$ be given. By \lp{AET} we can find a point $y_0$ which is uniformly recurrent and proximal to $t_0$ in the system $(\mathcal{X},T)$. Then applying \lp{eAET} to the point $x_0:=t_0$, $y_0:=y_0$, $x_0:=t_1(y_0)$, yields a point $y_1$ such that $(y_0,y_1)$ is uniformly recurrent and proximal $(x_0,x_1)$ in the $2$\nobreakdash-fold product of $(\mathcal{X},T)$. Iterating this process yields $(y_i)_{i<n}$ as needed.

\section{Minimal idempotent ultrafilters on countable algebras}\label{sec:mincount}

A countable algebra $\mathcal{A}=\{A_1,A_2,\dots \}$ is a sequence of sets closed under intersections, unions and complement. A downward translation algebra is an algebra $\mathcal{A}$ which is additionally closed under downward translations, i.e.
\[
X\in\mathcal{A} \Rightarrow \Forall{n\in\Nat} \left(X-n \in \mathcal{A}\right)
.\]
(Membership in $\F$ is given by an index set of indexes $A_n$ such that $A_n\in \F$. Thus, this definition makes sense in \ls{RCA_0}.)

A (partial) \emph{non-principal ultrafilter} $\mathcal{F}$ for $\mathcal{A}$ is a subset of $\mathcal{A}$ that satisfies the ultrafilter axioms relativized to $\mathcal{A}$, i.e.
\begin{align*}
   &\Forall{X,Y\in \mathcal{A}} \left(X \in \F \AND X\subseteq Y \IMPL Y\in \F\right) \\
  \AND\, &\Forall{X,Y\in \mathcal{F}} \left((X\cap Y)\in\F\right)  \\
  \AND\, &\Forall{X\in\mathcal{F}} \left(\Forall{n}\Exists{k>n} k\in X\right) \\
  \AND\, & \Forall{X\in \mathcal{A}} \left(X\in\F \OR \overline{X}\in\F\right).
\end{align*}

A \emph{partial idempotent ultrafilter} $\mathcal{F}$, (also \emph{downward translation partial ultrafilter}) is a filter that satisfies the non-principal ultrafilter axioms relativized to $\mathcal{A}$ 
and the following relativized idempotency condition
\[
\Forall{X\in\F} \left(\{\, n\mid X-n\in \F \,\} \in \F\right). \label{eq:idem} \tag{*}
\]
Note that we do not know whether $\{n\in \Nat \mid X-n \in \F\}$ is contained in $\mathcal{A}$. To overcome this problem we actually only check for an subset of $\{n\in \Nat \mid X-n \in \F\}$ that is contained in $\mathcal{A}$.

To keep the quantifier complexity low we will code membership in $\F$ by a monotone arithmetical predicate, i.e., an arithmetical formula $\phi$ such that i.e.\ $X\in\F:\equiv \phi(X)$ and $\phi(X) \AND X'\supseteq X \IMPL \phi(X')$.  (This definition is made relative to $\ls{ACA_0}$, since otherwise the membership property of $\F$ is not decidable.)
A partial non-principal ultrafilter in the sense of the above definition is then given by $\{\, X\in\mathcal{A} \mid \phi(X) \,\}$.

In \cite{jH04} Hirst considered a weaker form  of idempotent ultrafilters so called \emph{almost downward translation invariant ultrafilters} where \eqref{eq:idem} is replaced by the following
\begin{equation}\label{eq:hirst}
\Forall{X\in \F} \Exists{n} \left(X-n \in \F\right)
.\end{equation}
Since $X\cap \{\, n\mid X-n\in \F \,\}\in \F$ for a partial idempotent ultrafilter and each set in $\F$ is infinite and therefore nonempty, this condition is satisfied by any partial idempotent ultrafilter. 

A (partial) \emph{minimal (idempotent) partial ultrafilter} for $\mathcal{A}$ is an (idempotent) partial ultrafilter for $\mathcal{A}$ which additionally satisfies
\begin{equation}\label{eq:min}\tag{\dag}
\Forall{X\in \F} \{\, n \mid X-n\in \F\,\}\text{ is syndetic}.
\end{equation}

\begin{theorem}\label{thm:npu}
  Over \ls{RCA_0} the following statements are equivalent.
  \begin{enumerate}[label=(\roman*)]
  \item\label{enum:npu:1} For every countable algebra $\mathcal{A}$ there exists a partial non-principal ultrafilter.
  \item\label{enum:npu:2} For every countable algebra $\mathcal{A}$ there exists a partial minimal ultrafilter.
  \item\label{enum:npu:3} \ls{ACA_0}.
  \end{enumerate}
\end{theorem}
\begin{proof}
  It is clear that it suffices to show $\ref{enum:npu:1} \Rightarrow \ref{enum:npu:3}$ and $\ref{enum:npu:3} \Rightarrow \ref{enum:npu:2}$.

  \noindent
  $\ref{enum:npu:1} \Rightarrow \ref{enum:npu:3}$: We show that $\ref{enum:npu:1}$ proves $\Pi^0_1$-comprehension. Let $\phi_0(n,x)$ be a quantifier-free formula. 
  Let $\mathcal{A}$ be the algebra given by the sets $X_n:= \{ x \mid \Forall{x' \le x} \phi_0(n,x') \}$. Let $\mathcal{F}$ be a partial non-principal ultrafilter for $\mathcal{A}$. We have for any $n$ that
  \begin{align*}
    \Forall{x} \phi_0(n,x) &\IFF X_n \text{ is infinite} \\
    & \IFF X_n\in \F
  \end{align*}
  Thus, the characteristic function of $X_n\in \F$ is also a comprehension function for $\Forall{x} \phi_0(n,x)$.

  \noindent
  $\ref{enum:npu:3} \Rightarrow \ref{enum:npu:2}$:
  Let $\mathcal{A}=\{A_1,A_2,\dots\}$ be given. We will interpret each set as a point in the Cantor-space $2^{\Nat}$ given by the characteristic function 
  \[
  \chi_{A_i}(n) =
  \begin{cases}
    0 & \text{if $n\in A_i$}, \\
    1 & \text{if $n\notin A_i$.}
  \end{cases}
  \]
  We will use the shift $Tu(n)\mapsto u(n+1)$ as transformation. With this $(2^\Nat, T)$ becomes a compact topological dynamical system.
  To treat all sets simultaneously we will use the $\Nat$-fold product of this system and arrive at the system $({(2^\Nat)}^\Nat, T^\Nat)$.
  Let 
  \begin{equation}\label{eq:pointx}
    x:=
    \begin{pmatrix}
      \chi_{A_1} \\
      \chi_{A_2} \\
      \vdots 
    \end{pmatrix} \in {(2^\Nat)}^\Nat
  \end{equation}
  be the point coding all sets in the algebra.

  By Lemma 5.5 of \cite{BHS87} and \prettyref{lem:min} we can find in \ls{ACA_0} a $y= \begin{psmallmatrix}
    y_1 \\
    y_2 \\[-1ex]
    \vdots 
  \end{psmallmatrix}\in \orbc(x)$ such that $\orbc(y)$ is minimal. We claim that the filter $\F$ given by 
  \[
  A_n \in \F \quad \text{if{f}} \quad y_n(0) = 0
  \]
  is a minimal partial ultrafilter.
  To show that it is a partial ultrafilter note that, let $(n_j)_{j\in\Nat}$ be a strictly increasing sequence such that  $\lim_{j\to\infty} T^{n_j}x = y$. In particular, for each $A_i\in \mathcal{A}$ the limit $\lim_{j\to\infty} T^{n_j} \chi_{A_i}$ exists. Now either $A_i$ or $\overline{A_i}$ is in $\mathcal{F}$, since either $\lim_{j\to\infty} T^{n_j} \chi_{A_i}=0$ or $\lim_{j\to\infty} T^{n_j} \chi_{\overline{A_i}}=0$. In the same way one can show that $\F$ is closed under taking supersets and finite intersections.

  To show that the filter $\F$ is minimal, i.e., it satisfies \eqref{eq:min}, consider an $A_i\in\mathcal{F}$. Since $A_i\in \F$ we have that $y_i(0)=0$. Let $\epsilon>0$ be small enough such that such that for two points $x'=  \begin{psmallmatrix}
    x'_1 \\
    x'_2 \\[-1ex]
    \vdots 
  \end{psmallmatrix}, x'' = 
   \begin{psmallmatrix}
    x''_1 \\
    x''_2 \\[-1ex]
    \vdots 
  \end{psmallmatrix} \in {(2^\Nat)}^\Nat$ we have that $d(x',x'')<\epsilon$ implies $x'_i(0)=x''_i(0)$ for that given $i$. Since the orbit closure of $y$ is minimal, repeated applications of \prettyref{lem:forb} to the open set $U:=\{ x \mid d(x,y) < \epsilon \}$ give that
  \begin{equation}\label{eq:psyn}
  \{ n'\in\Nat \mid (T^{n'} y_i)(0) = 0 \}\text{ is syndetic.}
  \end{equation}
  Since $T$ is continuous we have that
  \begin{align*}
  \lim_{j\to\infty} (T^{n_j}\chi_{A_i-n'}) &= \lim_{j\to\infty} T^{n'}( T^{n_j}\chi_{A_i}) 
  \\ & = T^{n'}\left(\lim_{j\to \infty} T^{n_j} \chi_{A_i}\right)
  \\ & = T^{n'} y_i .
  \end{align*}
  Combining this with \eqref{eq:psyn} we get that
  \[
  \{ n'\in\Nat \mid A_i-n' \in\F \}\text{ is syndetic}
  \]
  and thus \eqref{eq:min}.
\end{proof}

\begin{lemma}\label{lem:aetmin}
  The Auslander-Ellis theorem (\lp{AET}) proves that for each countable downward translation algebra $\mathcal{A}$ there exists a partial minimal idempotent ultrafilter.
\end{lemma}
\begin{proof}
  Let $\mathcal{A}=\{A_1,A_2,\dots\}$. We consider the point coding all sets of $\mathcal{A}$ as in \eqref{eq:pointx}.
  
  By \lp{AET} there exists a point $y = \begin{psmallmatrix}
    y_1 \\
    y_2 \\[-1ex]
    \vdots 
  \end{psmallmatrix}$ which is uniformly recurrent and proximal to $x$.
  By \prettyref{pro:ip} there exists an increasing sequence $(n_j)_j$ such that 
  \[
  \iplim_{n\to \FS{(n_j)_j}} T^n \chi_{A_i} = y_i
  .\]
  
  We set 
  \[
  \F := \left\{\, X \subseteq \Nat \sizeMid \iplim_{{n\to \FS{(n_j)_j}}} (T^n \chi_X) \text{ exists and } \iplim_{{n\to \FS{(n_j)_j}}} (T^n \chi_X)(0) = 0\,\right\}
  \]
  or in other words
  \[
  \F = \left\{\, X \subseteq \Nat \sizeMid \Exists{k} \FS{(n_j)_{j=k}^\infty} \subseteq X \,\right\}
  .\]
  It is clear that $\F$ can be defined by a monotone, arithmetical formula and it is straightforward to check that $\F$ forms a filter containing only infinite sets, i.e.\ it is closed under finite intersections and taking supersets.

  The filter $\F$ satisfies \eqref{eq:idem} by Lemma 3.3 in \cite{aK15}. It satisfies \eqref{eq:min} by the same argument as in the proof of \prettyref{thm:npu}. Thus, $\F$ is a minimal idempotent partial ultrafilter for $\mathcal{A}$. 
\end{proof}

\begin{theorem}\label{thm:m}
  Over \ls{ACA_0} the following statements are equivalent.
  \begin{enumerate}
  \item\label{enum:m:1} The Auslander-Ellis theorem \lp{AET}.
  \item\label{enum:m:2} The iterated Hindman's theorem \lp{IHT}.
  \item\label{enum:m:3} For every countable downward translation algebra there exists a partial minimal idempotent ultrafilter.
  \item\label{enum:m:4} For every countable downward translation algebra there exists a partial idempotent ultrafilter.
  \item\label{enum:m:5} For every countable downward translation algebra there exists an almost downward translation invariant ultrafilter (in the sense of Hirst \cite{jH04}).
  \end{enumerate}
\end{theorem}
\begin{proof}
  $\ref{enum:m:2}\Rightarrow \ref{enum:m:1}$ follows from \cite{BHS87}. $\ref{enum:m:5}\Rightarrow \ref{enum:m:2}$ follows from \cite{jH04}. $\ref{enum:m:3}\Rightarrow \ref{enum:m:4}\Rightarrow \ref{enum:m:5}$ is clear and $\ref{enum:m:1}\Rightarrow \ref{enum:m:3}$ is \prettyref{lem:aetmin}.
\end{proof}

\section{Minimal idempotent ultrafilters in a higher-order setting}\label{sec:minhigh}

In this section we will work in the higher-order systems \ls{RCA_0^\omega}, \ls{ACA_0^\omega} corresponding to \ls{RCA_0} and \ls{ACA_0}. We refer the reader to \cite{uK05b} for an introduction to these systems and assume that he is familiar with the treatment of ultrafilters in these systems in \cite{aK12b,aK15}.

The statement that a minimal idempotent ultrafilter exists can be formulated in \ls{RCA_0^\omega} in the following way.
\[
  \lpf{\Umin}\colon\left\{
  \begin{aligned}
    \Exists{\U^2} \big(\ &\Forall{X^1} \left(X\in\U \OR \overline{X}\in\U\right) \\
    \AND\, &\Forall{X^1,Y^1} \left(X \cap Y \in \U \IMPL Y\in \U\right) \\
    \AND\, &\Forall{X^1,Y^1} \left(X,Y\in \U \IMPL (X\cap Y)\in\U\right)  \\
    \AND\, &\Forall{X^1} \left(X\in\U \IMPL \Forall{n}\Exists{k>n} (k\in X)\right) \\
    \AND\, &\Forall{X^1} \left(X\in\U \IMPL \left\{\, n\in\Nat \mid X - n \in \U \,\right\} \in \U \right) \\
    \AND\, &\Forall{X^1} \left(\left\{\, n\in\Nat \mid X - n \in \U \,\right\}\text{ is syndetic}\right) \\
    \AND\, &\Forall{X^1} \left(\U(X) =_0 \sg(\U(X)) =_0 \U(\lambda n. \sg(X(n)))\right)\big)
  \end{aligned}
  \right.
\]
The first four lines state that $\U$ is a non-principal ultrafilter. The fifth line indicates that $\U$ is idempotent and the sixth that it is minimal.
The last line states that $\U$ respects coding of sets as characteristics functions.

In \cite{aK15} we showed that the existence of idempotent ultrafilters \lpf{\Uidem} is $\Pi^1_2$\nobreakdash-\hspace{0pt}conservative over $\ls{RCA_0^\omega} + \lp{IHT}$. This theorem is proved by replacing the ultrafilter occurring in a proof of a $\Pi^1_2$ statement by a finite sequence of partial idempotent ultrafilters\footnote{In \cite{aK15} the filters are called downward translation partial ultrafilter} $(\mathcal{F}_i)_{i<n}$ for an increasing sequence of countable algebras $(\mathcal{A}_i)_{i<n}$ such that each $\F_{i+1}$ refines $\F_i$, in the sense that
\[
  \Forall{i} \Forall{j<i} \left(\mathcal{F}_i \cap \mathcal{A}_j = \mathcal{F}_j\right)
.\]

We indicate here how to change the construction of a downward translation partial ultrafilter $\F$ such that additionally for each set $X\in \F$ we have that 
\begin{equation}\notag
\left\{\, n\in\Nat \mid X - n \in \U \,\right\}\text{ is syndetic},
\end{equation}
and thus that they can be used to replace \emph{minimal} idempotent ultrafilters. 

Like in the proof of \prettyref{lem:aetmin} and in \cite{aK15} the partial idempotent ultrafilters will be of the following form.
\[
\F((n_i)_i) := \{\, X \mid \Exists{m} \FS{(n_i)_{i=m}^\infty}\subseteq X\,\}
.\]
Note that the filter here is not given by a formula anymore but by a higher-order object. To construct this object from the sequence $(n_i)$ one in general needs $\mu$.

The construction in the proof of \prettyref{lem:aetmin} can be summed up in the following way. Let $\mathcal{A}$ be an countable algebra and $x\in {2^\Nat}^\Nat$ the point in the system $\left({(2^\Nat)}^\Nat, T^\Nat\right)$ corresponding to $\mathcal{A}$ via \eqref{eq:pointx}. Then any uniformly recurrent point $y$ proximal to $x$ gives rise of partial minimal idempotent ultrafilter of the form $\F((n_i)_i)$.
The following lemma give a reversal to the construction.

We will write $\ulim{\F}{}$ for the limit along a filter $\F$, i.e.~$\ulim{\F}{n} x_n = x$ if for all $\epsilon$ the set $\{ n \mid d(x_n,x)< \epsilon \}$ is contained in $\F$.

\begin{lemma}\label{lem:aetrev}
  Let $\mathcal{A}=\{A_1,A_2,\dots\}$ be an countable algebra and let $\F$ be a partial minimal idempotent ultrafilter.

  For $x$ as in \eqref{eq:pointx} we have that
  \[
  \ulim{\F}{n} {(T^\Nat)}^n(x) = 
  \begin{pmatrix}
    y_1 \\ y_2 \\ \vdots
  \end{pmatrix}=: y \text{ exists}
  \]
  and that $y$ is uniformly recurrent and proximal to $x$ and for all $i$ we have that $A_i\in \F$ if{f} $y_i(0)=0$.
\end{lemma}
\begin{proof}
  For each set $X$ we have that 
  \[
  \ulim{\F}{n} {(T^\Nat)}^n(\chi_X)\text{ exists and equals $0$ } \quad\text{if{f}}\quad X\in \F
  .\]
  Thus $\ulim{\F}{n} {(T^\Nat)}^n(x)=y$ exists and $A_i\in \F$ if{f} $y_i(0)=0$.

  Since the sets $O_{i,k,b} := \big\{\, x \in {(2^\Nat)}^\Nat \,\big\vert\, x_i(k)=b \,\big\}$ form a subbase of the topology of ${(2^\Nat)}^\Nat$ it suffices to show that
  \[
  \{\, n \mid T^{n} y_i(k) = y_i(k) \,\}\text{ is syndetic for each $i,k$},
  \]
  to obtain that $y$ is uniformly recurrent.

  This follows from 
  \begin{align*}
    (T^{n} y_i)(k)  & = (T^{k+n} y_i)(0) = \left(\ulim{\F}{n'} {T}^{k+n+n'}(\chi_{A_i})\right)(0) \\
    & =
    \begin{cases}
      0 & \text{if } (A_i - k) - n \in \F, \\
      1 & \text{if } \overline{(A_i - k) - n} \in \F.
    \end{cases}
  \end{align*}
  For simplicity we assume that $(A_i - k) \in \F$ then by the minimality of $\F$  we have $T^{n} y_i(k)=0=y_i(k)$ for a syndetic sets of $n$. The case of $(A_i - k) \notin \F$ is similar.
  
  Now we show that $x,y$ are proximal. Let $X_{\epsilon} := \left\{ n \sizeMid  d(T^n(x),y) < \epsilon \right\}$. By assumption $X_\epsilon\in \F$ for each $\epsilon$. 
  Fix $\epsilon>0$. By the properties of $\F$ we have 
  \[
  X' := X_\epsilon \cap \left\{ n \sizeMid X_\epsilon - n \in \F \right\}\in \F.
  \]
  Let $m$ be an arbitrary element of $X'$ and choose $\delta>0$ such that $d(z,y)<\delta$ implies $d(T^mz,T^my) < \epsilon$.
  Now let $n\in X_\delta \cap X' \in \F$ be such that $n>m$. Then $d(T^nx,y)<\delta$; thus $d(T^{m+n}x,T^my) < \epsilon$ and $m+n\in X_\epsilon$. We get
  \[
  d(T^m x, T^m y) \le d(T^{m}x, y) + d(y,T^{m+n}x) + d(T^{m+n}x,T^{m}y) < 3\epsilon
  .\qedhere
  \]
\end{proof}

In particular, the above construction shows that there is a one-to-one correspondence between points $y$, that are uniformly recurrent and proximal to $x$ as in~\eqref{eq:pointx}, and partial minimal idempotent ultrafilters. Moreover, this construction shows that each partial idempotent ultrafilter $\F$ on an algebra $\mathcal{A}$ is equal to an partial idempotent ultrafilter of the form $\F((n_i))$ in the sense that $\F \cap \mathcal{A} =  \F((n_i)) \cap \mathcal{A}$.

Using this one-to-one correspondence one obtains the following lemma.
\begin{lemma}\label{lem:15}
  $\ls{RCA_0^\omega} + \lpf{\mu} + \lp{eAET}$ proves the following.
  Given a countable algebra $\mathcal{A}=\{A_1,A_2,\dots\}$ and a partial minimal idempotent ultrafilter $\F$. Then for each countable algebra $\mathcal{A'}$ extending $\mathcal{A}$ there exists a sequence $(n_i)_i$ such that $\F((n_i)_i)$ is a partial minimal idempotent ultrafilter and $\F((n_i)_i)\cap \mathcal{A} = \F$.
\end{lemma}
\begin{proof}
  By \prettyref{lem:aetrev} there is a point $y_1 \in {(2^\Nat)}^\Nat$ that is uniformly recurrent and proximal to the point $x_1$ as in \eqref{eq:pointx} (for the algebra $\mathcal{A}$) and with $(y_1)_i(0)=0 \IFF A_i\in \F$. Now let $x_2$ be as in \eqref{eq:pointx} now for the algebra $\mathcal{A}'$ then by the \lp{eAET} there exists an $y_2$ such that $\begin{pmatrix} y_1 \\ y_2 \end{pmatrix}$ is uniformly recurrent and proximal to $\begin{pmatrix} x_1 \\ x_2 \end{pmatrix}$. By the construction in \prettyref{lem:aetmin} there exists an partial minimal idempotent ultrafilter $\F=\F((n_i)_i)$ for the algebra $\mathcal{A'}$. Since the membership of $X\in\mathcal{A}$ only depends on $y_1$ we have that $\F((n_i)_i)\cap \mathcal{A} = \F$.
\end{proof}

Replacing Theorem~3.5 in \cite{aK15} with the previous lemma one can now show the following variant of Theorem~2.6 of \cite{aK15}.
\begin{theorem}\label{thm:conserv}
  The system $\ls{ACA_0^\omega} + \lpf{\mu} + \lp{IHT} + \lpf{\Umin}$ is $\Pi^1_2$\nobreakdash-conservative over $\ls{ACA_0^\omega} + \lp{eAET}$.
\end{theorem}
\begin{proof}
  One first notes that the Section~4 of \cite{aK15} goes through unchanged since \lpf{\Umin} differs from \lpf{\Uidem} only by 
  \[
  \left\{\, n\in\Nat \mid X - n \in \U \,\right\}\text{ is syndetic}
  \]
  which is arithmetic and can be made quantifier free using \lpf{\mu}.

  Now one just replaces Theorem~3.5 in \cite{aK15} in the construction of the approximation of the ultrafilter with \prettyref{lem:15}.
\end{proof}

Since the sequence of algebras is directly given as terms in \cite{aK15} the principle \lp{eAET'} suffices to carry out the above construction and one obtains.

\begin{corollary}\label{col:conserv}
  The system $\ls{ACA_0^\omega} + \lpf{\mu} + \lp{IHT} + \lpf{\Umin}$ is $\Pi^1_2$\nobreakdash-conservative over $\ls{ACA_0^\omega} + \lp{eAET'}$.
\end{corollary}

\begin{theorem}
  $\ls{ACA_0^\omega} + \lpf{\mu} + \lpf{\Umin} \vdash \lp{eAET'}$.
\end{theorem}
\begin{proof}
  Let $(t_i)_{i<k}$ be as in the definition of \lp{eAET'} and $\U$ be a minimal idempotent ultrafilter. By Theorem~19.26 of \cite{HS12},
  $\ulim{\U}{n} T^n(t_0) =: y_0$ is uniformly recurrent and proximal to $t_0$. Since $\U$ is a minimal idempotent ultrafilter also 
  \[
  \ulim{\U}{n} T^n
  \begin{pmatrix}
    t_0 \\
    t_1(y_0)
  \end{pmatrix}
  = 
  \begin{pmatrix}
    y_0 \\
    y_1
  \end{pmatrix}
  \]
  exists and is again uniformly recurrent (now in $(\mathcal{X} \times \mathcal{X}, T \times T)$) and proximal to $t_0,t_1(y_0)$. 
  Iterating this construction yields a solution to \lp{eAET'}.
\end{proof}

\begin{corollary}[to \prettyref{thm:conserv}]
  
\prettyref{thm:conserv} remains true if one replaces the system by the following.
\begin{multline*}
  \ls{ACA_0^\omega} + \lpf{\mu} + \lp{IHT} \vdash \\
  \Forall{f} \big(\Exists \U \,\text{\rm [$\U$ is a minimal idempotent ultrafilter extending $t_\F(f)$]} \IMPL \Exists{g} \lf{A}(f,g)\big)
  \end{multline*}
  where $t_\F$ is a closed term such that $t_\F(f)$ codes a partial minimal idempotent ultrafilter. (Cf.~Remark~4.4 of \cite{aK15}.)
\end{corollary}
\begin{proof}
  Use \prettyref{lem:aetrev} to obtain the first approximation of the ultrafilter. Then continue using \prettyref{lem:15} as in the proof of \prettyref{thm:conserv}.
\end{proof}
\begin{theorem}
  Over $\ls{ACA_0^\omega}+\lpf{\mu}$  the statement 
  \[
  \Forall{f} \big(\Exists \U \,\text{\rm [$\U$ is a minimal idempotent ultrafilter extending $t_\F(f)$]}
  \] 
  for a suitable term $t_\F$ proves \lp{eAET}.
\end{theorem}
\begin{proof}
  Let $x_1,y_1,x_2$ be given as in the definition of \lp{eAET} then by \prettyref{pro:ip} and the proof of \prettyref{lem:aetmin} there exists an increasing sequence $(n_i)_i$ such that 
  \[
  \iplim_{n\to\FS{(n_i)_i}} T^n x_1 = y_1
  \]
  and 
  \[
  \F = \left\{\, X\subseteq \Nat \sizeMid \Exists{k} \FS{(n_i)_{i=k}^\infty} \subseteq X \,\right\}.
  \]
  is a partial minimal idempotent ultrafilter for the algebra generated by $\{\, n \in \Nat \mid d(T^n x_1,y_1) < 2^{-k} \,\}$. It is easy to see that $\F$ is definable from $x_1,y_1$ using $\mu$.
  For a minimal idempotent ultrafilter $\U$ extending $\F$ we have then
  \[
  \ulim{\U}{n} T^n 
  \begin{pmatrix} 
    x_1 \\ x_2 
  \end{pmatrix} = 
  \begin{pmatrix} 
    y_1 \\ y_2 
  \end{pmatrix}
  \]
  and by Theorem~19.26 of \cite{HS12} the point  $\begin{pmatrix} y_1 \\ y_2 \end{pmatrix}$ is uniformly recurrent and proximal to $\begin{pmatrix} x_1 \\ x_2 \end{pmatrix}$.
\end{proof}

\begin{remark}
  Like in \cite[Remark~4.3]{aK15} all the previous results on minimal idempotent ultrafilter over $\Nat$ also apply to minimal idempotent ultrafilters over any other countable semigroup $G$. The proofs are formulated such that neither  commutativity nor any other specify property of $\Nat$ has been used. However, for the construction of the partial minimal idempotent ultrafilter one then need the Auslander-Ellis theorem for this particular semigroup.

  For \lp{AET} the analysis Blass, Hirst, Simpson shows that \lp{AET} for a semigroup $G$ is equivalent to \lp{ACA_0} plus \lp{IHT} for this group which is, if $G$ is not trivial, equivalent to \lp{IHT} for $\Nat$. For \lp{eAET},\lp{eAET'} we do not where their variant for different groups is equivalent to $G$. However, given the previous fact we conjecture that they are also equivalent.
\end{remark}

\begin{figure}
  \[
  \xymatrix @C=0.4cm @R=0.7cm{ & 
    \txt{
      on countable \\ algebras
    } & 
    \txt{
      as  higher-order object \\
      ($\Pi^1_2$-consequences)} \\
    & \save []+<1.6cm,0cm>* \txt{\ls{ACA_0^+}} \ar[d]_{\ref{f4}} 
    \restore & \txt{ } \\
    \text{minimal idempotent ultrafilters$^{\ref{f1}}$} & \lp{AET} \ar @{<->}[d]_(0.3){\ref{f4},\ref{f5}} & \lp{eAET}, \lp{eAET'} \ar[d] \ar[l]\\
    \text{idempotent ultrafilters$^{\ref{f2}}$} & \lp{IHT} \ar@{->}[d]_(0.32){\ref{f4}} \ar@{=}[r]& \lp{IHT} \ar[d] 
    \\
    \text{non-principal ultrafilters$^{\smash[t]{\ref{f3}}}$} & \lp{ACA_0}\ar@{=}[r]  & \lp{ACA_0}  
    \ar @{.} "2,1"+<-2.5cm,-0.55cm>;"2,3"+<1.9cm,-0.55cm>
    \ar @{.} "3,1"+<-2.5cm,-0.55cm>;"3,3"+<1.9cm,-0.55cm>
    \ar @{.} "4,1"+<-2.5cm,-0.55cm>;"4,3"+<1.9cm,-0.55cm>
    \ar @{.} "1,2"+<-1.2cm,0.5cm>;"4,2"+<-1.2cm,-1.8cm>
    \ar @{.} "1,2"+<1.4cm,0.5cm>;"2,2"+<1.4cm,1.6ex>
    \ar @{.} "2,2"+<1.4cm,-1.4ex>;"4,2"+<1.4cm,-1.8cm>
  }
  \]
  \begin{flushleft}\small
    Arrows indicate all known implications and non-implications except for $\lp{ACA_0^+} \nrightarrow \lp{ACA_0}$.
    
    \smallskip
    \begin{enumerate*}[label=(\arabic*)]
    \item\label{f1} \prettyref{thm:m}, \prettyref{thm:conserv} and discussion below,
    \item\label{f2} \cite{aK15},
    \item\label{f3} \prettyref{thm:npu}, \cite{aK12b}, see also \cite{hT14},
    \item\label{f4} \cite{BHS87},
    \item\label{f5} \prettyref{cor:aetiht}.
    \end{enumerate*}
  \end{flushleft}
  \vspace{-2ex}
  \caption{}
\end{figure}

\section{Discussion and Questions}

\begin{question}
  What is the strength of \lp{eAET} and \lp{eAET'}?  
\end{question}

A set $X\subseteq \Nat$ is called \emph{central} if one of the following equivalent conditions holds.
\begin{enumerate}
\item There exists a compact topological dynamical system $(\mathcal{X},T)$ and points $x,y\in\mathcal{X}$ with $x,y$ are proximal and $y$ is uniformly recurrent (in other words they satisfy the conclusion of \lp{AET}) such that 
\[
X = \{ n \mid d(T^n x,y) < \epsilon \}
\]
for an $\epsilon$.
\item $X$ is an element of a minimal idempotent ultrafilter.
\item $X$ is syndetic and an IP-set (i.e.~contains a set of the form $\FS{Y}$ for an infinite set $Y$).
\end{enumerate}
see \cite[Chap.~8 \S 3]{hF81} for the original definition and \cite{HS12} for this equivalences.

It is easy to see that \lp{AET} proves that each finite partition of $\Nat$ contains a central set, see \cite[Theorem~8.8]{hF81}.
In the same way \lp{eAET} proves that each finite partition of a central set contains a central set (in other words being central is partition stable), see Remark at the end of Chap.~8 \S 3 of \cite{hF81}. It is open whether \lp{AET} is sufficient for this.

\begin{question}
  What is the strength of the statement that each finite partition of a central set contains a central set?
\end{question}

\bibliographystyle{amsplain}
\bibliography{../bib}

\end{document}